\documentclass[11pt]{amsart}
\pdfoutput=1

\usepackage{amssymb,mathtools}
\usepackage{stmaryrd}
\usepackage{microtype}

\usepackage[hmargin=1.25in,vmargin=1in]{geometry}

\usepackage[numbers]{natbib}
\renewcommand\bibfont\footnotesize

\usepackage{enumitem}

\usepackage{aliascnt}
\usepackage{hyperref}
\hypersetup{
  pdftitle={Invariant connections, Lie algebra actions, and
    foundations of numerical integration on manifolds},
  pdfauthor={Hans Z. Munthe-Kaas, Ari Stern, and Olivier Verdier},
  pdfsubject={MSC 2010: 53C05, 17A30, 17B66},
  bookmarksopen=false,
}

\theoremstyle{plain} 
\newtheorem{theorem}{Theorem}[section]
\newaliascnt{lemma}{theorem}
\newtheorem{lemma}[lemma]{Lemma}
\aliascntresetthe{lemma}
\newaliascnt{corollary}{theorem}
\newtheorem{corollary}[corollary]{Corollary}
\aliascntresetthe{corollary}
\newaliascnt{proposition}{theorem}
\newtheorem{proposition}[proposition]{Proposition}
\aliascntresetthe{proposition}

\theoremstyle{definition}
\newtheorem{definition}[theorem]{Definition}

\theoremstyle{remark}
\newaliascnt{example}{theorem}
\newtheorem{example}[example]{Example}
\aliascntresetthe{example}
\newtheorem{remark}[theorem]{Remark}

\newcommand\csum{\sum_\circlearrowleft}

\begin{document}

\title[Connections, Lie algebra actions, and numerical integration]{Invariant
  connections, Lie algebra actions, and foundations of numerical integration
  on manifolds}
\author{Hans Z.\ Munthe-Kaas}
\address{Department of Mathematics, University of Bergen}
\email{hans.munthe-kaas@uib.no}

\author{Ari Stern}
\address{Department of Mathematics and Statistics,
  Washington University in St.~Louis}
\email{stern@wustl.edu}

\author{Olivier Verdier}
\address{Department of Computing, Mathematics and Physics, Western Norway University of Applied Sciences, and Department of Mathematics, KTH Royal Institute of Technology}
\email{olivier.verdier@hvl.no, olivierv@kth.se}

\begin{abstract}
  Motivated by numerical integration on manifolds, we relate the
  algebraic properties of invariant connections to their geometric
  properties. Using this perspective, we generalize some classical
  results of Cartan and Nomizu to invariant connections on
  algebroids. This has fundamental consequences for the theory of
  numerical integrators, giving a characterization of the spaces on
  which Butcher and Lie--Butcher series methods, which generalize
  Runge--Kutta methods, may be applied.
\end{abstract}

\subjclass[2010]{53C05, 17A30, 17B66}

\maketitle

\section{Introduction}

\subsection{Background and motivation}

A connection on a smooth manifold $M$ can be viewed as a
non-associative product on the Lie algebra of vector fields
$ \mathfrak{X} (M) $. The properties of the resulting algebra contain
geometric information about the connection and about $M$ itself. In
particular, a flat and torsion-free connection gives
$ \mathfrak{X} (M) $ the structure of a \emph{pre-Lie algebra}, while
a flat connection with parallel torsion gives $ \mathfrak{X} (M) $ the
structure of a \emph{post-Lie algebra}. The notion of a pre-Lie
algebra originates from work of
\citet{Vinberg1963,Gerstenhaber1963,AgGa1981}, while post-Lie algebras
are due to \citet{Vallette2007}.

Recently, \citet{MuLu2013} related this algebraic perspective to
certain analytical techniques for approximating flows of vector
fields: Butcher series methods
(\citet{Butcher1963,Butcher1972,HaWa1974}) in the pre-Lie case and
Lie--Butcher series methods
(\citet{Munthe-Kaas1995,Munthe-Kaas1998,Munthe-Kaas1999}) in the more
general post-Lie case. These techniques, which involve expressing
flows as formal power series in rooted trees and forests, were
originally developed for the analysis of numerical integrators.

It is natural to ask which manifolds $M$ admit such structures, and to
which the techniques of Butcher and Lie--Butcher series may therefore
be applied. \citet{Nomizu1954}, following earlier work of E.~Cartan
\citep{Cartan1927}, showed that $M$ admits a flat and torsion-free
connection (i.e., an affine manifold structure) if and only if it is
locally representable as an abelian Lie group with its canonical
affine connection, while $M$ admits a flat connection with parallel
torsion if and only if it is locally representable as a Lie group with
its $ ( - ) $-connection. It follows that Butcher series methods may
be applied in the former case and Lie--Butcher series methods in the
latter case.

However, a 1999 paper of \citet{Munthe-Kaas1999} (see also
\citet{MuWr2008}) showed that such methods may be applied, more
generally, whenever a Lie algebra $ \mathfrak{g} $ acts transitively
on $M$. This includes not only the case where $ M = G $ is the Lie
group integrating $ \mathfrak{g} $, but also (for example) when $M$ is
a homogeneous space, or when $M$ is equipped with a frame of vector
fields generating a Lie subalgebra
$ \mathfrak{g} \subset \mathfrak{X} (M) $. In fact, $M$ need not admit
a flat connection at all, as in the example of
$ \mathfrak{g} = \mathfrak{so}(3) $ acting transitively on
$ M =\mathbb{S} ^2 $, the $2$-sphere. The Cartan--Nomizu
characterization is therefore not the end of the story.

To include these examples in the algebraic framework, \citet{MuLu2013}
considered connections more general than affine connections: namely,
connections on a \emph{Lie algebroid} $ A \rightarrow M $, which is an
anchored vector bundle with a compatible Lie bracket on the space of
sections $ \Gamma (A) $. (An affine connection is just the special
case when $ A = T M $ is the tangent bundle with the Jacobi--Lie
bracket.) When equipped with a connection such that $ \Gamma (A) $ is
a pre-Lie algebra or post-Lie algebra, we say that $A$ is a
\emph{pre-Lie algebroid} or \emph{post-Lie algebroid}. In particular,
a $ \mathfrak{g} $-action on $M$ has an associated \emph{action
  algebroid} $ \mathfrak{g} \ltimes M \rightarrow M $; as a vector
bundle, this is just the trivial bundle with fiber $ \mathfrak{g} $
over $M$. \citet{MuLu2013} showed that the canonical flat connection
makes $ \mathfrak{g} \ltimes M $ into a post-Lie algebroid, and when
$ \mathfrak{g} $ is abelian, this is in fact a pre-Lie algebroid.

This previous work therefore gives sufficient conditions for
$A \rightarrow M $ to admit a (pre-Lie) post-Lie structure: it is
sufficient for $A$ to be an (abelian) action algebroid. The purpose of
the present work is to prove conditions that are \emph{both necessary
  and sufficient}, thereby giving a full characterization of the
spaces to which Butcher and Lie--Butcher series methods may be used
for numerical integration and analysis of flows on $M$.

\subsection{Overview}

The main results of this paper,
\autoref{thm:pre-Lie_iff_abelian_action} and
\autoref{thm:post-Lie_iff_action}, show that if $A \rightarrow M $ is
transitive, i.e., the anchor is surjective, then it admits a (pre-Lie)
post-Lie structure if and only if it is locally isomorphic to the
action algebroid of some transitive (abelian) $ \mathfrak{g} $-action
with its canonical flat connection, and this isomorphism is global
when $M$ is simply connected. (In the non-transitive case, this result
holds leaf-by-leaf on the foliation induced by the anchor.) This
generalizes the Cartan--Nomizu results stated above, which correspond
to the special case $ A = T M $.

Consequently, the \emph{only} way to apply Butcher and Lie--Butcher
series methods to manifolds, at least locally, is the one introduced
twenty years ago: equip the manifold with a transitive Lie algebra
action.

The paper is organized as follows:

\begin{itemize}
\item \autoref{sec:connections} begins by introducing a purely
  algebraic treatment of connections, relating Lie-admissible,
  pre-Lie, and post-Lie algebras of connections to curvature and
  torsion. We then bring geometry into the picture by applying this
  framework to algebras of affine connections on $M$, linking the
  pre-Lie and post-Lie conditions to the results of \citet{Cartan1927}
  and \citet{Nomizu1954}.

\item \autoref{sec:algebroids} considers connections on anchored
  bundles and Lie algebroids and gives necessary and sufficient
  conditions, in terms of curvature and torsion, for an algebroid to
  be Lie-admissible, pre-Lie, or post-Lie.

\item \autoref{sec:action} proves the main results, applying the
  framework of the previous sections to characterize transitive
  pre-Lie and post-Lie algebroids in terms of transitive
  $\mathfrak{g}$-actions on $M$. We also remark on the non-transitive
  case, in which these results hold leaf-by-leaf on the foliation of
  $M$ induced by the anchor, and compare our results to those of
  \citet{Blaom2006} and \citet{AbCr2012}, who drop the transitivity
  assumption but require strictly stronger conditions on the
  connection than the pre-Lie and post-Lie conditions.
\end{itemize} 

\section{Algebras of invariant connections}
\label{sec:connections}

In this section, we begin by considering a purely algebraic notion of
a connection as a non-associative product on a Lie algebra.  We then
recall the definitions of Lie-admissible, pre-Lie, and post-Lie
algebras, and we discuss the relationship between these algebras and
the curvature and torsion of the connection corresponding to the
product. Finally, we apply this framework to affine connections,
obtaining necessary and sufficient conditions for $M$ to admit a
connection giving $ \mathfrak{X} (M) $ a Lie-admissible, pre-Lie, or
post-Lie structure, and relating this to the Cartan--Nomizu
classification.

\subsection{Connections on Lie algebras}
\label{sec:lieAlgebraConnections}

Certain properties of the connections we wish to study are purely
algebraic, in the sense that they do not depend on any local or
geometric arguments. Therefore, we postpone geometry to subsequent
sections and begin in the following algebraic setting.

\begin{definition}
  Let $ \bigl( L , \llbracket \cdot , \cdot \rrbracket \bigr) $ be a
  Lie algebra over a field $ \Bbbk $ of characteristic zero. A
  \emph{connection} on $L$ is a $ \Bbbk $-linear map
  $ \nabla \colon L \rightarrow \operatorname{End} (L) $,
  $ X \mapsto \nabla _X $.\footnote{By $ \operatorname{ End } (L) $,
    we mean linear endomorphisms on $L$ as a vector space over
    $ \Bbbk $, not necessarily Lie algebra endomorphisms.}
  Equivalently, a connection corresponds to a $\Bbbk$-bilinear product
  $ \triangleright $ on $L$ defined by
  $ X \triangleright Y \coloneqq \nabla _X Y $.
\end{definition}

The Lie bracket on $L$ makes it possible to define algebraic notions
of curvature and torsion, which are formally identical to the familiar
definitions from differential geometry.

\begin{definition}
  Given a connection $ \nabla $ on
  $ \bigl( L , \llbracket \cdot , \cdot \rrbracket \bigr) $, its
  \emph{curvature} is the $ \Bbbk $-bilinear map
  $ R \colon L \times L \rightarrow \operatorname{ End } (L) $ given
  by
  \begin{equation}
    \label{eqn:R}
    R(X,Y) \coloneqq \nabla _X \nabla _Y - \nabla _Y \nabla _X - \nabla _{ \llbracket X, Y \rrbracket } ,
  \end{equation}
  and its \emph{torsion} is the $ \Bbbk $-bilinear map
  $ T \colon L \times L \rightarrow L $ given by
  \begin{equation}
    \label{eqn:T}
    T ( X, Y ) \coloneqq \nabla _X Y - \nabla _Y X - \llbracket X, Y \rrbracket .
  \end{equation}
  If $ R = 0 $, the connection is \emph{flat}, and if $ T = 0 $, it is
  \emph{torsion-free}.
\end{definition}

\begin{remark}
  A representation of a Lie algebra is precisely a flat connection.
\end{remark}

Covariant derivatives $ \nabla R $ and $ \nabla T $ are defined by the
usual product rules,
\begin{align}
  (\nabla _Z R) (X,Y) W
  &\coloneqq \nabla _Z \bigl( R ( X, Y ) W \bigr)  \label{eq:NR} \\
  &\quad  - R ( \nabla _Z X, Y ) W - R ( X, \nabla _Z Y ) W - R ( X, Y ) \nabla _Z W , \notag \\
  ( \nabla _Z T ) ( X, Y ) &\coloneqq \nabla _Z \bigl( T ( X, Y) \bigr) - T ( \nabla _Z X , Y ) - T ( X , \nabla _Z Y ) \label{eq:NT}.
\end{align}
 The curvature (resp., torsion) is \emph{parallel}
if $ \nabla R = 0 $ (resp., $ \nabla T = 0 $), and if both the
curvature and torsion are parallel, we say that $ \nabla $ is an
\emph{invariant connection}.

Associated to each $ \nabla $ is a \emph{dual connection}\footnote{For
  connections on a Lie algebroid, this is the notation used by
  \citet{CrFe2003}; the name \emph{dual connection} appears in
  \citet{Blaom2006}, who denotes it by $ \nabla ^\ast $.}
$ \overline{ \nabla } _X Y \coloneqq \nabla _Y X + \llbracket X, Y
\rrbracket $, which is seen to satisfy
$ \overline{ \overline{ \nabla } } = \nabla $. The curvature and
torsion of $ \overline{ \nabla } $ are denoted by $ \overline{ R } $
and $ \overline{ T } $. Observe that $T$ and $ \overline{ T } $ are related by
\begin{equation}
  \label{eqn:T_Tbar}
  T ( X, Y ) = \nabla _X Y - \overline{ \nabla } _X Y = - \overline{ T } ( X, Y ),
\end{equation} 
so the torsion expresses the difference between the primal and dual
connections, and a connection is its own dual if and only if it is
torsion-free. In particular, the connection
$ \widetilde{ \nabla } \coloneqq \frac{1}{2} ( \nabla + \overline{
  \nabla } ) $ is always torsion-free.

\begin{example}
  On any Lie algebra, we may define the trivial connection
  $ \nabla _X Y = 0 $, which has the dual connection
  $ \overline{ \nabla } _X Y = \llbracket X, Y \rrbracket $. We see
  that $ R = 0 $ trivially and $ \overline{ R } = 0 $ by the Jacobi
  identity, and indeed $ \nabla $ and $ \overline{ \nabla } $ are Lie
  algebra representations: the trivial representation and adjoint
  representation, respectively.
\end{example}

\begin{proposition}\label{prop:covtor}
  For a connection $ \nabla $ on a Lie algebra, we have
  \begin{equation}
    \label{eqn:bianchi1}
    (\nabla _Z T)(X,Y)
    =\overline{ R } ( X, Y)  Z + R (Y,Z) X + R(Z,X) Y. 
  \end{equation} 
\end{proposition}

\begin{proof}
  Consider the three terms defining $ ( \nabla _Z T ) ( X, Y) $ in
  \eqref{eq:NT}.  First,
  \begin{align*} 
    \nabla _Z \bigl(T ( X, Y )\bigr)
    &= \nabla _Z \nabla _X Y - \nabla _Z \nabla _Y X - \nabla _Z \llbracket X, Y \rrbracket \\
    &= \nabla _Z \nabla _X Y - \nabla _Z \nabla _Y X - \overline{ \nabla } _{ \llbracket X, Y \rrbracket } Z - \bigl\llbracket Z, \llbracket X, Y \rrbracket \bigr\rrbracket .
  \end{align*}
  For the second term,
  \begin{align*}
    T(\nabla _Z X, Y )
    &= \nabla _{ \nabla _Z X } Y - \nabla _Y \nabla _Z X - \llbracket \nabla _Z X , Y \rrbracket \\
    &= \overline{ \nabla } _Y \nabla _Z X - \nabla _Y \nabla _Z X \\
    &= \overline{ \nabla } _Y \bigl( \overline{ \nabla } _X Z + \llbracket Z, X \rrbracket \bigr) - \nabla _Y \nabla _Z X \\
    &= \overline{ \nabla } _Y \overline{ \nabla } _X Z - \nabla _Y \nabla _Z X + \nabla _{\llbracket Z, X \rrbracket} Y  + \bigl\llbracket Y, \llbracket Z, X \rrbracket \bigr\rrbracket ,
  \end{align*}
  and likewise for the third,
  \begin{equation*}
    T(X, \nabla _Z Y )
    = \nabla _X \nabla _Z Y - \overline{ \nabla } _X \overline{ \nabla } _Y Z + \nabla _{ \llbracket Y, Z \rrbracket } X + \bigl\llbracket X, \llbracket Y, Z \rrbracket \bigr\rrbracket .
  \end{equation*} 
  Combining these and applying the Jacobi identity gives
  \eqref{eqn:bianchi1}.
\end{proof}

\begin{corollary}
  \label{cor:Rbar_NT}
  Assuming $R = 0 $, we have $ \overline{ R } = 0 $ if and only if
  $ \nabla T = 0 $.
\end{corollary}

Cyclic sums of trilinear functions will appear repeatedly. We denote
\begin{equation*}
  \csum f (X,Y,Z) \coloneqq f (X,Y,Z) + f(Y,Z,X) + f(Z,X,Y).
\end{equation*}
For example, the Jacobi identity may be written as
$ \csum \bigl\llbracket X, \llbracket Y, Z \rrbracket \bigr\rrbracket
$.

\begin{proposition}\label{prop:cyclicT} 
  For a connection $ \nabla $ on a Lie algebra, we have
  \begin{equation}
    \label{eqn:bianchi2}
    \csum T \bigl( X, T(Y,Z) \bigr) = \csum R ( X, Y ) Z  + \csum\overline{R} ( X, Y ) Z .
  \end{equation}
\end{proposition}

\begin{proof}
  From \eqref{eqn:T_Tbar}, we have
  \begin{align*}
    T \bigl( Z, T ( X,Y) \bigr)
    &= \nabla _Z \bigl(T ( X,Y )\bigr) + \overline{ \nabla } _Z\bigl( \overline{ T } ( X, Y )\bigr) \\
    &= \nabla _Z \nabla _X Y - \nabla _Z \nabla _Y X - \overline{ \nabla } _{ \llbracket X, Y \rrbracket } Z - \bigl\llbracket Z, \llbracket X, Y \rrbracket \bigr\rrbracket \\
    &\quad + \overline{ \nabla } _Z \overline{ \nabla} _X Y - \overline{ \nabla } _Z \overline{ \nabla } _Y X -  \nabla _{ \llbracket X, Y \rrbracket } Z - \bigl\llbracket Z, \llbracket X, Y \rrbracket \bigr\rrbracket .
  \end{align*}
  The last equality comes from the expression obtained for
  $\nabla _Z \bigl(T ( X,Y )\bigr)$ in the previous proof, together
  with the corresponding version for $ \overline{ \nabla } $. Taking
  the cyclic sum of both sides and applying the Jacobi identity gives
  \eqref{eqn:bianchi2}.
\end{proof}

\begin{corollary}
  \label{cor:T_liebracket}
  If $ R = \overline{ R } = 0 $, then $ T $ is a Lie
  bracket.
\end{corollary}

\begin{proof}
  By definition, $T$ is always bilinear and skew-symmetric. If
  $ R = \overline{ R } = 0 $, the right-hand side of
  \eqref{eqn:bianchi2} vanishes, so $T$ also satisfies the Jacobi
  identity.
\end{proof}

\begin{corollary}[Bianchi's first identity]
  For a connection $ \nabla $ on a Lie algebra, we have
  \begin{equation}
    \csum (\nabla_X T) (Y,Z) = \csum R(X,Y)Z+\csum T\bigl(X,T(Y,Z)\bigr)\label{eq:B1}
  \end{equation} 
\end{corollary}

\begin{proof}
  Take the cyclic sum of both sides of \eqref{eqn:bianchi1} and apply
  \eqref{eqn:bianchi2}.
\end{proof}

\begin{remark}
  Bianchi's second identity,
  \begin{equation*}
    \csum (\nabla_XR)(Y,Z) =\csum R \bigl( X,T(Y,Z) \bigr)  ,
  \end{equation*} 
  also holds in this setting. The proof is a lengthy calculation which
  we omit, since we will not need this identity in the sequel.
\end{remark}

\subsection{Lie-admissible, pre-Lie, and post-Lie algebras of connections}

We now consider Lie-admissible, pre-Lie, and post-Lie algebraic
structures on an algebra $ ( \mathcal{A}, \triangleright ) $. For each
of these, we show that the product $ \triangleright $ may be
interpreted as a connection on a Lie algebra, and we characterize
these algebraic structures in terms of the curvature and torsion of
this connection.

The \emph{associator} of the product $ \triangleright $ is denoted by
\begin{equation*}
  a (X, Y, Z ) \coloneqq X \triangleright ( Y \triangleright Z ) - ( X \triangleright Y ) \triangleright Z ,
\end{equation*}
following the sign convention of \citet{MuLu2013}. In the sequel, an
important role is played by the \emph{associator triple bracket},
\begin{equation*}
  [X,Y,Z] \coloneqq a (X, Y, Z ) - a (Y, X, Z ) .
\end{equation*}
When $ \triangleright $ corresponds to a connection on a Lie algebra,
the following useful identity relates the associator triple bracket to
curvature and torsion.

\begin{proposition}
  \label{prop:triple}
  For a connection $ \nabla $ on a Lie algebra, we have
  \begin{equation*}
    [X,Y,Z] = R(X,Y) Z - T ( X, Y ) \triangleright Z .
  \end{equation*} 
\end{proposition}

\begin{proof}
  By definition of $R$ and $T$ and the linearity of the connection,
  \begin{align*}
    R(X,Y) Z &= X \triangleright ( Y \triangleright Z ) - Y \triangleright ( X \triangleright Z ) - \llbracket X, Y \rrbracket \triangleright Z ,\\
    T(X,Y) \triangleright Z &= ( X \triangleright Y ) \triangleright Z - ( Y \triangleright X ) \triangleright Z - \llbracket X, Y \rrbracket \triangleright Z ,
  \end{align*}
  so subtracting gives $ [ X, Y, Z ] $.
\end{proof}

\subsubsection{Lie-admissible algebras}

The definition of a Lie-admissible algebra is due to
\citet{Albert1948}.

\begin{definition}
  An algebra $ ( \mathcal{A}, \triangleright ) $ is
  \emph{Lie-admissible} if $ \csum [X,Y,Z] = 0 $.
\end{definition}

Such algebras are called ``Lie-admissible'' due to the following equivalence.

\begin{proposition}
  \label{prop:lieadmissible_bracket}
  $ ( \mathcal{A}, \triangleright ) $ is Lie-admissible if and only if
  the commutator bracket
  $ \llbracket X, Y \rrbracket \coloneqq X \triangleright Y - Y
  \triangleright X $ is a Lie bracket on $ \mathcal{A} $.
\end{proposition}

\begin{proof}
  The commutator bracket is always skew-symmetric and bilinear.  A
  short calculation shows that
  $ \csum \bigl\llbracket X, \llbracket Y, Z \rrbracket
  \bigr\rrbracket = \csum [X,Y,Z] $, so the Jacobi identity is
  equivalent to the Lie-admissibility condition.
\end{proof}
  
\begin{proposition}\label{prop:lieadm}
  The following are equivalent:
  \begin{enumerate}[label=(\roman*)]
  \item $ ( \mathcal{A}, \triangleright ) $ is a Lie-admissible algebra.
  \item
    $ \bigl( \mathcal{A}, \llbracket \cdot , \cdot \rrbracket \bigr) $
    is a Lie algebra with a torsion-free connection $ \nabla $.
  \end{enumerate} 
\end{proposition}

\begin{proof}
  $ T (X, Y ) = 0 $ says precisely that
  $ \llbracket X, Y \rrbracket = X \triangleright Y - Y \triangleright
  X $.
\end{proof}

\subsubsection{Pre-Lie algebras}

The notion of pre-Lie algebra appears in work of \citet{Vinberg1963}
in differential geometry and \citet{Gerstenhaber1963} in algebra. They
also appear in the work of \citet{AgGa1981} in control theory, under
the name ``chronological algebras.''

\begin{definition}
  An algebra $ ( \mathcal{A} , \triangleright ) $ is \emph{pre-Lie} if
  $ [X, Y, Z ] = 0 $.
\end{definition}

It follows immediately from this definition that every pre-Lie algebra
is a Lie-admissible algebra, so $ \triangleright $ corresponds to a
torsion-free connection $ \nabla $ on
$ \bigl( \mathcal{A}, \llbracket \cdot , \cdot \rrbracket \bigr) $,
where $ \llbracket \cdot , \cdot \rrbracket $ is the commutator
bracket. The next result shows that the pre-Lie condition corresponds
to the case where $ \nabla $ is also flat.

\begin{proposition}\label{prop:preLie}
  The following are equivalent:
  \begin{enumerate}[label=(\roman*)]
  \item $ ( \mathcal{A}, \triangleright ) $ is a pre-Lie algebra.
  \item
    $ \bigl( \mathcal{A}, \llbracket \cdot , \cdot \rrbracket \bigr) $
    is a Lie algebra with a flat and torsion-free connection $ \nabla $.
  \end{enumerate} 
\end{proposition}

\begin{proof}
  If $ ( \mathcal{A}, \triangleright ) $ is pre-Lie, then
  \autoref{prop:lieadm} says that $ \triangleright $ corresponds to a
  torsion-free connection, and \autoref{prop:triple} with $ T = 0 $
  gives $ R(X,Y) Z = [X,Y,Z] = 0 $, so the connection is also
  flat. The converse direction is immediate from \autoref{prop:triple}
  with $ R = 0 $ and $ T = 0 $.
\end{proof}

\subsubsection{Post-Lie algebras}

The notion of post-Lie algebra is due to \citet{Vallette2007}.

\begin{definition}
  A \emph{post-Lie algebra}
  $ \bigl( \mathcal{A}, [ \cdot , \cdot ] , \triangleright \bigr) $ is
  a Lie algebra $ \bigl( \mathcal{A} , [ \cdot , \cdot ] \bigr) $
  equipped with a product $ \triangleright $ satisfying the
  compatibility conditions
  \begin{subequations}
    \label{eqn:postLie}
    \begin{align}
      X \triangleright [ Y, Z ] &= [ X \triangleright Y, Z ] + [ X , Y \triangleright Z ] , \label{eqn:postLieTorsion} \\
      [X,Y] \triangleright Z &=  [X,Y,Z] \label{eqn:postLieCurvature}.
    \end{align}
\end{subequations}
\end{definition}

Given a post-Lie algebra
$ \bigl( \mathcal{A}, [ \cdot , \cdot ] , \triangleright \bigr) $, we
immediately see from \eqref{eqn:postLieCurvature} that
$ ( \mathcal{A}, \triangleright ) $ is pre-Lie if and only if
$ [X,Y] \triangleright Z = 0 $ for all $ X, Y, Z \in \mathcal{A}
$. Furthermore, any pre-Lie algebra
$ ( \mathcal{A}, \triangleright ) $ admits a post-Lie structure by
taking $ [ \cdot , \cdot ] $ to be trivial.

\begin{proposition}
  \label{prop:postLieBracket}
  If $ \bigl( \mathcal{A}, [ \cdot , \cdot ] , \triangleright \bigr) $
  is a post-Lie algebra, then
  \begin{equation}
    \llbracket X, Y \rrbracket \coloneqq X \triangleright Y - Y \triangleright X + [ X, Y ] \label{eq:postLieJac}
  \end{equation}
  is also a Lie bracket on $\mathcal{A}$.
\end{proposition}

\begin{proof}
  This bracket is always skew-symmetric and bilinear. To establish the
  Jacobi identity for $ \llbracket \cdot , \cdot \rrbracket $, a
  calculation shows that
  \begin{multline}\label{eqn:postLieAlgebraBianchi}
    \csum \bigl\llbracket X, \llbracket Y, Z \rrbracket \bigr\rrbracket = \csum \bigl( X \triangleright [ Y, Z ] - [ X \triangleright Y, Z ] - [ Y, X \triangleright Z ] \bigr) \\
    + \csum \bigl( [X,Y,Z] - [X,Y] \triangleright Z \bigr) + \csum
    \bigl[ X ,[Y,Z] \bigr] .
  \end{multline}
  On the right-hand side, the first cyclic sum vanishes by
  \eqref{eqn:postLieTorsion}, the second by
  \eqref{eqn:postLieCurvature}, and the last by the Jacobi identity
  for $ [ \cdot , \cdot ] $.
\end{proof}

Assuming
$ \bigl( \mathcal{A}, [ \cdot , \cdot ] , \triangleright \bigr) $ is
post-Lie, we consider $ \triangleright $ as a connection on
$ \bigl( \mathcal{A}, \llbracket \cdot , \cdot \rrbracket \bigr) $.
It follows from \eqref{eq:postLieJac} that
$ T ( X ,Y ) = - [ X, Y ] $.  Therefore, the post-Lie condition
\eqref{eqn:postLieTorsion} says that $ \nabla T = 0 $, while
\eqref{eqn:postLieCurvature} says that $ R = 0 $ (using
\autoref{prop:triple} to relate the triple bracket to curvature and
torsion). Furthermore, the vanishing of
\eqref{eqn:postLieAlgebraBianchi} corresponds to the first Bianchi
identity \eqref{eq:B1}.

Conversely, if
$ \bigl( \mathcal{A} , \llbracket \cdot , \cdot \rrbracket \bigr) $ is
a Lie algebra with connection $ \nabla $, we may define
$ [ X, Y ] = - T(X,Y) $ and ask when
$ \bigl( \mathcal{A}, [ \cdot , \cdot ] , \triangleright \bigr) $ is
post-Lie. The following result shows that the conditions
$ \nabla T = 0 $ and $ R = 0 $ are sufficient, as well as necessary.

\begin{proposition}\label{prop:postLie}
  Let $ \triangleright $, $ [ \cdot , \cdot ] $, and
  $ \llbracket \cdot , \cdot \rrbracket $ be related by
  \eqref{eq:postLieJac}. Then the following are equivalent:
  \begin{enumerate}[label=(\roman*)]
  \item
    $ \bigl( \mathcal{A} , [ \cdot , \cdot ] , \triangleright \bigr) $
    is a post-Lie algebra.

  \item
    $ \bigl( \mathcal{A}, \llbracket \cdot , \cdot \rrbracket \bigr) $
    is a Lie algebra with a flat, parallel-torsion connection $ \nabla $.

  \item
    $ \bigl( \mathcal{A}, \llbracket \cdot , \cdot \rrbracket \bigr) $
    is a Lie algebra with a flat connection $ \nabla $ and flat dual
    connection $ \overline{ \nabla } $.
  \end{enumerate} 
\end{proposition}

\begin{proof}
  We have already shown, in discussion above, that (i) implies (ii),
  and \autoref{cor:Rbar_NT} says that (ii) and (iii) are
  equivalent. To show that (ii) implies (i), observe that
  $ \nabla T = 0 $ and $ R = 0 $ immediately give
  \eqref{eqn:postLieTorsion} and \eqref{eqn:postLieCurvature}, while
  \autoref{cor:T_liebracket} implies that $ [ \cdot , \cdot ] = - T $
  is a Lie bracket.
\end{proof}

\subsection{Algebras of affine connections}
\label{sec:affine}

We now bring geometry into the picture by considering affine
connections. The main result of this section,
\autoref{th:affineconnect}, gives necessary and sufficient conditions
for $M$ to admit a connection giving $ \mathfrak{X} (M) $ a
Lie-admissible, pre-Lie, or post-Lie structure. \citet{MuLu2013} had
previously shown sufficiency but not necessity of these conditions.

Recall that a vector field $ X \in \mathfrak{X} (M) $ on a smooth
manifold defines a derivation $ f \mapsto X [f] $ on
$ C ^\infty (M) $. This forms a Lie algebra
$ \bigl( \mathfrak{X} (M) , \llbracket \cdot , \cdot \rrbracket _J
\bigr) $ with respect to the \emph{Jacobi--Lie bracket},
\begin{equation*}
  \llbracket X, Y \rrbracket _J [f] \coloneqq X \bigl[ Y [f] \bigr] - Y \bigl[ X [f] \bigr] .
\end{equation*} 
An \emph{affine connection} is not only $ \mathbb{R} $-bilinear on
$ \mathfrak{X} (M) $, but is $ C ^\infty (M) $-linear in the first
argument and satisfies a Leibniz rule in the second,
\begin{equation*}
  \nabla _{ f X } Y = f \nabla _X Y , \qquad \nabla _X f Y = X[f] Y + f \nabla _X Y ,
\end{equation*}
for all $ f \in C ^\infty (M) $, $ X, Y \in \mathfrak{X} (M) $. It is
straightforward to show that if $ \nabla $ is an affine connection,
then so are $ \overline{ \nabla } $ and $ \widetilde{ \nabla } $, but
we postpone the proof to the more general setting of Lie algebroids,
in \autoref{sec:algebroids}.  The curvature and torsion of an affine
connection $ \nabla $ are defined with respect to
$ \llbracket \cdot , \cdot \rrbracket _J $, and these definitions
imply that $R$ and $T$ are \emph{tensorial}, i.e.,
$ C ^\infty (M) $-linear in all arguments.

To apply the framework developed in this section to affine
connections, we first show that the brackets
$ \llbracket \cdot , \cdot \rrbracket $ constructed for
Lie-admissible, pre-Lie, and post-Lie algebras agree with the
Jacobi--Lie bracket.

\begin{lemma}\label{lem:LeibnizImpliesJacobiLie}
  If $ \llbracket \cdot , \cdot \rrbracket $ is a Lie bracket on
  $ \mathfrak{X} (M) $ satisfying the Leibniz rule
  \begin{equation}
    \label{eqn:vectorFieldLeibniz}
    \llbracket X, f Y \rrbracket = X [f] Y + f \llbracket X, Y \rrbracket ,
  \end{equation}
  for all $ f \in C ^\infty (M) $, $ X, Y \in \mathfrak{X} (M) $, then
  $ \llbracket \cdot , \cdot \rrbracket = \llbracket \cdot , \cdot
  \rrbracket _J $.
\end{lemma}

\begin{proof}
  Using the Jacobi identity and Leibniz rule, a calculation gives
  \begin{align*}
    0 &= \bigl\llbracket X, \llbracket Y, fZ \rrbracket \bigr\rrbracket + \bigl\llbracket Y, \llbracket f Z, X \rrbracket \bigr\rrbracket + \bigl\llbracket fZ, \llbracket X, Y \rrbracket \bigr\rrbracket \\
      &= X \bigl[ Y [ f] \bigr] Z - Y \bigl[ X [f] \bigr] Z - \llbracket  X, Y \rrbracket [f] Z \\
      &= \bigl( \llbracket X, Y \rrbracket _J - \llbracket X, Y \rrbracket \bigr) [f] Z, 
  \end{align*}
  and the result follows since $ f \in C ^\infty (M) $,
  $ X, Y, Z \in \mathfrak{X} (M) $ are arbitrary.
\end{proof}

\begin{proposition}
  Let $ \nabla $ be an affine connection and $ [ \cdot , \cdot ] $ a
  tensorial bracket.  If
  $ \bigl( \mathfrak{X} (M) , [ \cdot , \cdot ] , \triangleright
  \bigr) $ is post-Lie, then the bracket
  $ \llbracket X, Y \rrbracket \coloneqq X \triangleright Y - Y
  \triangleright X + [ X, Y ] $ of \eqref{eq:postLieJac} agrees with
  the Jacobi--Lie bracket.
\end{proposition}

\begin{proof}
  It suffices to check that the Leibniz rule
  \eqref{eqn:vectorFieldLeibniz} holds:
  \begin{align*}
    \llbracket X, f Y \rrbracket
    &= X \triangleright (f Y) - (f Y) \triangleright X + [ X, f Y ] \\
    &= X[f] Y + f ( X \triangleright Y ) - f ( Y \triangleright X ) + f [X,Y]\\
    &= X[f] Y + f \llbracket X, Y \rrbracket .
  \end{align*}
  The result then follows by \autoref{lem:LeibnizImpliesJacobiLie}.
\end{proof}

Repeating the same computation with $ [ \cdot , \cdot ] = 0 $, we find
the following.

\begin{proposition}
  \label{prop:affine_lie-admissible_jacobi-lie}
  Let $ \nabla $ be an affine connection. If
  $ \bigl( \mathfrak{X} (M) , \triangleright \bigr) $ is
  Lie-admissible, then the commutator bracket
  $ \llbracket X, Y \rrbracket \coloneqq X \triangleright Y - Y
  \triangleright X $ agrees with the Jacobi--Lie bracket.
\end{proposition}

Applying \autoref{prop:lieadm}, \autoref{prop:preLie}, and
\autoref{prop:postLie} now yields our main result on algebras of
affine connections.

\begin{theorem}\label{th:affineconnect}
  Let $ \nabla $ be an affine connection on a smooth manifold $M$.
  \begin{enumerate}[label=(\roman*)]
  \item $ \bigl( \mathfrak{X} (M) , \triangleright \bigr) $ is a
    Lie-admissible algebra if and only if $ \nabla $ is torsion-free.\label{th:affineconnect_lie-admissible}
  \item $ \bigl( \mathfrak{X} (M) , \triangleright \bigr) $ is a
    pre-Lie algebra if and only if $ \nabla $ is flat and
    torsion-free.\label{th:affineconnect_pre-lie}
  \item
    $ \bigl( \mathfrak{X} (M) , [ \cdot , \cdot ] , \triangleright
    \bigr) $ is a post-Lie algebra, with $ [ \cdot , \cdot ] $ being
    tensorial, if and only if $ \nabla $ is flat with parallel torsion
    $ T = - [ \cdot , \cdot ] $.\label{th:affineconnect_post-lie}
  \end{enumerate} 
\end{theorem}

Every smooth manifold $M$ admits an affine connection $ \nabla $, and
thus a torsion-free connection $ \widetilde{ \nabla } $, so
Lie-admissibility of
$ \bigl( \mathfrak{X} (M) , \triangleright \bigr) $ reveals nothing
about $M$. By contrast, the other two algebraic structures are deeply
associated with special geometries classified by \citet{Cartan1927}
and \citet{Nomizu1954}.

\begin{corollary} Let $M$ be a smooth manifold.
  \label{cor:nomizu}
  \begin{itemize}
  \item $M$ admits an affine connection $ \nabla $ such that
    $ \bigl( \mathfrak{X} (M) , \triangleright \bigr) $ is pre-Lie if
    and only if $M$ is locally representable as an abelian Lie group
    with its canonical affine connection.

  \item $M$ admits a connection $ \nabla $ and a tensorial bracket
    $ [ \cdot , \cdot ] $ such that
    $ \bigl( \mathfrak{X} (M) , [ \cdot , \cdot ] , \triangleright
    \bigr) $ is post-Lie if and only if $M$ is locally representable
    as a Lie group with its $ ( - ) $-connection.
  \end{itemize} 
\end{corollary}

\begin{proof}
  Combine \autoref{th:affineconnect} with results (a) and (b) stated in \S 20 of
  \citet{Nomizu1954}.
\end{proof}

\section{The geometry and algebra of connections on Lie algebroids}
\label{sec:algebroids}

In this section, we recall how connections may be generalized from the
tangent bundle of $M$ (i.e., affine connections) to more general
anchored bundles and Lie algebroids\footnote{Generally, the results of
  this section also hold for \emph{Lie--Rinehart algebras}, which are
  an algebraic abstraction of Lie algebroids (e.g., replacing smooth
  functions on $M$ by a commutative algebra, vector fields on $M$ by
  derivations on that algebra, etc.). However, since we are laying the
  groundwork for \autoref{sec:action}, which \emph{does} use the
  smooth manifold structure, we have chosen to use the language of Lie
  algebroids throughout.} over $M$. We then characterize connections
inducing Lie-admissible, pre-Lie, and post-Lie structures in terms of
their curvature and torsion, generalizing the results of
\autoref{sec:affine} for affine connections.

\subsection{Lie algebroids}

\citet{Pradines1967} is credited for introducing Lie algebroids, which
simultaneously generalize tangent bundles and Lie algebras (among many
other things). A comprehensive treatment is given by
\citet{Mackenzie2005}.

\begin{definition}
  An \emph{anchored bundle} $ ( A, \rho ) $ is a vector bundle
  $ A \rightarrow M $ with a vector bundle morphism
  $ \rho \colon A \rightarrow T M $ called the \emph{anchor map}. A
  \emph{Lie algebroid}
  $ \bigl( A, \rho , \llbracket \cdot , \cdot \rrbracket \bigr) $ is
  an anchored bundle equipped with a Lie bracket
  $ \llbracket \cdot , \cdot \rrbracket $ on the space of sections
  $ \Gamma (A) $, satisfying the Leibniz rule
  \begin{equation}
    \label{eqn:lieAlgebroidLeibniz}
    \llbracket X, f Y \rrbracket = \rho (X) [f] Y + f \llbracket X, Y \rrbracket ,
  \end{equation}
  for all $ f \in C ^\infty (M) $, $ X, Y \in \Gamma (A) $. We say
  that an anchored bundle or Lie algebroid is \emph{transitive} if the
  anchor is surjective.
\end{definition}

\begin{example}
  The tangent bundle $ A = T M $ is a Lie algebroid over $M$, with
  $ \rho $ the identity map on $ T M $ and
  $ \llbracket \cdot , \cdot \rrbracket $ the Jacobi--Lie
  bracket.

  More generally, any involutive distribution
  $ \mathcal{D} \subset T M $ is a Lie algebroid over $M$, with
  $ \rho $ the inclusion $ \mathcal{D} \hookrightarrow T M $ and
  $ \llbracket \cdot , \cdot \rrbracket $ the restriction of the
  Jacobi--Lie bracket to $ \mathcal{D} $. (In fact, this describes a
  \emph{Lie subalgebroid} of the tangent Lie algebroid.)
\end{example}

\begin{example}
  \label{ex:bla}
  A Lie algebra is just a Lie algebroid over a single point, with
  trivial anchor $ \rho = 0 $.

  More generally, a Lie algebroid with trivial anchor is called a
  \emph{bundle of Lie algebras}: since $ \rho = 0 $,
  \eqref{eqn:lieAlgebroidLeibniz} implies that
  $ \llbracket \cdot , \cdot \rrbracket $ is tensorial, so for each
  $ x \in M $ we have a well-defined pointwise Lie bracket
  $ \llbracket \cdot , \cdot \rrbracket _x $ on the fiber $ A _x
  $. Note that the fibers need not be isomorphic as Lie algebras,
  i.e., $A$ need not be isomorphic to the trivial bundle of Lie
  algebras $ \mathfrak{g} \times M $ for any Lie algebra
  $ \mathfrak{g} $.
\end{example}

\begin{example}
  An \emph{action} (sometimes called an \emph{infinitesimal action})
  of a Lie algebra $ \mathfrak{g} $ on $M$ is a Lie algebra
  homomorphism $ \mathfrak{g} \rightarrow \mathfrak{X} (M) $,
  $ \xi \mapsto \xi _M $. The \emph{action algebroid}
  $ A = \mathfrak{g} \ltimes M $ is the trivial vector bundle
  $ \mathfrak{g} \times M \rightarrow M $, together with the anchor
  $ \rho ( \xi, x ) \coloneqq \xi _M (x) $ induced by the action of
  $\mathfrak{g}$ on $M$; the bracket
  $ \llbracket \cdot , \cdot \rrbracket $ is uniquely determined by
  \eqref{eqn:lieAlgebroidLeibniz} and the condition that it agrees
  with the bracket on $ \mathfrak{g} $ for constant sections.

  The action algebroid is transitive precisely when the
  $ \mathfrak{g} $-action is transitive. In particular, if $M$ is a
  homogeneous space, then its transitive Lie group action has a
  corresponding transitive Lie algebra action.
\end{example}

The following is a standard (yet important) property of Lie
algebroids. Some references on Lie algebroids, including
\citet{Mackenzie2005}, include this property as part of the definition
of a Lie algebroid, but this turns out to be redundant. (An
interesting account of this appears in the introduction to
\citet{Grabowski2003}.) The argument is essentially one due to
\citet{Herz1953}, later used by \citet[\S 6.1]{KoMa1990}.

\begin{proposition}
  \label{prop:lieAlgebroidAnchorHomomorphism}
  Given a Lie algebroid
  $ \bigl( A, \rho , \llbracket \cdot , \cdot \rrbracket \bigr) $ over
  $M$, the anchor map induces a Lie algebra homomorphism
  $ \bigl( \Gamma (A), \llbracket \cdot , \cdot \rrbracket \bigr)
  \rightarrow \bigl( \mathfrak{X} (M) , \llbracket \cdot , \cdot
  \rrbracket _J \bigr) $.
\end{proposition}

\begin{proof}
  Just as in the proof of \autoref{lem:LeibnizImpliesJacobiLie}, one
  uses the Jacobi identity together with the Leibniz rule
  \eqref{eqn:lieAlgebroidLeibniz} to calculate
  \begin{align*}
    0 &= \bigl\llbracket X, \llbracket Y, f Z \rrbracket \bigr\rrbracket + \bigl\llbracket Y, \llbracket f Z, X \rrbracket \bigr\rrbracket + \bigl\llbracket fZ, \llbracket X, Y \rrbracket \bigr\rrbracket \\
      &= \rho (X) \bigl[ \rho (Y) [f] \bigr] Z - \rho (Y) \bigl[ \rho (X) [f] \bigr] Z - \rho \bigl( \llbracket X, Y \rrbracket \bigr) [f] Z \\
      &= \Bigl( \bigl\llbracket \rho (X), \rho (Y) \bigr\rrbracket _J - \rho \bigl( \llbracket X, Y \rrbracket \bigr) \Bigr) [f] Z ,
  \end{align*}
  and the result follows since $ f \in C ^\infty (M) $,
  $ X, Y, Z \in \Gamma (A) $ are arbitrary.
\end{proof}

\begin{remark}
  \autoref{lem:LeibnizImpliesJacobiLie} is actually a special case of
  this result. In the language just introduced, the Leibniz rule
  \eqref{eqn:vectorFieldLeibniz} implies that
  $ \bigl( T M , \mathrm{id}_{TM}, \llbracket \cdot , \cdot \rrbracket
  \bigr) $ is a Lie algebroid, so $ \mathrm{id} _{ TM } $ induces a
  Lie algebra isomorphism between
  $ \bigl( \mathfrak{X} (M) , \llbracket \cdot , \cdot \rrbracket
  \bigr) $ and
  $ \bigl( \mathfrak{X} (M) , \llbracket \cdot , \cdot \rrbracket _J
  \bigr) $. That is,
  $ \llbracket \cdot , \cdot \rrbracket = \llbracket \cdot , \cdot
  \rrbracket _J $.
\end{remark}
  
An important consequence of this result is that the image of $\rho$
defines an involutive distribution on $M$, so there exists a
(generally singular) foliation of $M$ into leaves. The restriction to
each leaf $ L \subset M $ defines a transitive Lie algebroid over $L$.

\subsection{Connections, curvature, and torsion} We next discuss
connections, first on anchored bundles and then on Lie algebroids,
where the latter largely follows the treatment given in
\citet{Fernandes2002,CrFe2003}.

\begin{definition}
  Given an anchored bundle $ ( A, \rho ) $ over $M$, an
  \emph{$A$-connection} on a vector bundle $ E \rightarrow M $ is an
  $ \mathbb{R} $-bilinear map
  $ \Gamma (A) \times \Gamma (E) \rightarrow \Gamma (E) $,
  $ (X, u ) \mapsto \nabla _X u $, which is $ C ^\infty (M) $-linear
  in the first argument and satisfies a Leibniz rule in the second,
  i.e.,
  \begin{equation*}
    \nabla _{ f X } u = f \nabla _X u , \qquad \nabla _X f u = \rho (X) [f] u + f \nabla _X u ,
  \end{equation*}
  for all $ f \in C ^\infty (M) $.
\end{definition}

An affine connection is just a $ T M $-connection, where, as before,
the anchor is the identity map. Given a $ T M $-connection on an
anchored bundle $A$, the following construction gives an induced
$A$-connection on $A$.

\begin{proposition}
  \label{prop:tm-connection_a-connection}
  Let $ ( A, \rho ) $ be an anchored bundle over $M$ and $ \nabla $ be
  a $ T M $-connection on $A$. Then
  $ \nabla _X Y \coloneqq \nabla _{ \rho (X) } Y $ is an
  $A$-connection on $A$.
\end{proposition}

\begin{proof}
  $ \mathbb{R} $-bilinearity follows from the $\mathbb{R}$-bilinearity
  of the $ T M $-connection, together with the fact that $\rho$ is a
  vector bundle morphism. For any $ f \in C ^\infty (M) $ and
  $ X, Y \in \Gamma (A) $, we have
  \begin{equation*}
    \nabla _{ f X } Y = \nabla _{ \rho ( f X ) } Y = \nabla _{ f \rho (X) } Y = f \nabla _{ \rho (X) } Y = f \nabla _X Y ,
  \end{equation*}
  and
  \begin{equation*}
    \nabla _X f Y = \nabla _{ \rho (X) } f Y = \rho (X) [f] Y + f \nabla _{ \rho (X) } Y = \rho (X) [f] Y + f \nabla _X Y ,
  \end{equation*}
  which completes the proof.
\end{proof}

If $ \bigl( A , \rho, \llbracket \cdot , \cdot \rrbracket \bigr) $ is
a Lie algebroid and $ \nabla $ is an $A$-connection on $A$, then
$ \nabla $ is also a connection on the Lie algebra
$ \bigl( \Gamma (A) , \llbracket \cdot, \cdot \rrbracket \bigr) $, in
the sense of \autoref{sec:lieAlgebraConnections}. Therefore, all of
the results in that section immediately hold in the Lie algebroid
setting. We now show that
$ \overline{ \nabla } _X Y \coloneqq \nabla _Y X + \llbracket X, Y
\rrbracket $ and
$ \widetilde{ \nabla } \coloneqq \frac{1}{2} ( \nabla + \overline{
  \nabla } ) $ are in fact connections in the Lie algebroid sense, not
just the Lie algebra sense.

\begin{proposition}
  If $ \bigl( A, \rho , \llbracket \cdot , \cdot \rrbracket \bigr) $
  is a Lie algebroid, and if $ \nabla $ is an $A$-connection on $A$,
  then so are $ \overline{ \nabla } $ and $ \widetilde{ \nabla } $.
\end{proposition}

\begin{proof}
  $ \mathbb{R} $-bilinearity of $ \overline{ \nabla } $ follows from
  the $\mathbb{R}$-bilinearity of $ \nabla $ and of
  $ \llbracket \cdot , \cdot \rrbracket $. For any $ f \in C ^\infty (M) $ and $ X, Y \in \Gamma (A) $,
\begin{align*} 
  \overline{ \nabla } _{ f X } Y
  &= \nabla _Y f X + \llbracket f X, Y \rrbracket \\
  &= \rho (Y) [f] X + f \nabla _Y X - \rho (Y) [f]X + f
    \llbracket X, Y \rrbracket\\
  &= f \overline{ \nabla } _X Y,
\end{align*}
and
\begin{align*} 
  \overline{ \nabla } _X f Y
  &= \nabla _{ f Y } X + \llbracket X, f Y \rrbracket \\
  &= f \nabla _Y X + \rho (X) [f] Y + f \llbracket X, Y \rrbracket \\
  &= \rho (X) [f] Y + f \overline{ \nabla } _X Y ,
\end{align*}
so $ \overline{ \nabla } $ is an $A$-connection. That
$ \widetilde{ \nabla } $ is also an $A$-connection follows easily from
the fact that $ \nabla $ and $ \overline{ \nabla } $ are
$A$-connections.
\end{proof}

\begin{example}
  Let $ \mathfrak{g} $ be a Lie algebra, considered as a Lie algebroid
  over a single point. The trivial connection $ \nabla _\xi \eta = 0 $
  is a $ \mathfrak{g} $-connection on $\mathfrak{g}$, and
  $ \overline{ \nabla } _\xi \eta = \llbracket \xi, \eta \rrbracket =
  \operatorname{ad} _\xi \eta $. We can thus identify $ \nabla $ with
  the trivial representation and $ \overline{ \nabla } $ with the
  adjoint representation of $ \mathfrak{g} $ on itself.  This is
  readily generalized to the case where $A$ is a bundle of Lie
  algebras over $M$.
\end{example}

\begin{example}
  \label{ex:actionAlgebroid}
  Let $ A = \mathfrak{g} \ltimes M $ be an action algebroid. As a
  vector bundle, this is just the trivial bundle
  $ \mathfrak{g} \times M $, so we can define the obvious
  $ T M $-connection $ \nabla $ vanishing on constant
  sections. Identifying $ \xi, \eta \in \mathfrak{g} $ with the
  corresponding constant sections, it follows that the corresponding
  $A$-connections on $A$ satisfy ${\nabla} _\xi \eta = 0 $ and
  $  \overline{\nabla} _\xi \eta = \llbracket \xi, \eta \rrbracket $.

  In particular, if $ M = G $ is the Lie group integrating
  $ \mathfrak{g} $, then we may identify constant sections of
  $ \mathfrak{g} \ltimes G $ with left-invariant vector fields on $G$
  (and arbitrary sections with arbitrary vector fields). Under this
  identification, the connections $ \nabla $, $ \overline{ \nabla } $,
  and $ \widetilde{ \nabla } $ correspond, respectively, to the affine
  $ (-) $-, $ (+) $-, and $ (0) $-connections of \citet{CaSc1926}.
\end{example}

The curvature and torsion of an $A$-connection on a Lie algebroid $A$
are defined exactly as in \eqref{eqn:R}--\eqref{eqn:T}, and all the
results of \autoref{sec:lieAlgebraConnections} involving the curvature
and torsion of $ \nabla $, $ \overline{ \nabla } $, and
$ \widetilde{ \nabla } $ immediately hold in this setting.

As with affine connections, $R$ and $T$ are \emph{tensorial}, i.e.,
$ C ^\infty (M) $-linear in each argument, not just
$ \mathbb{R} $-linear, so they contain local, geometric information
about the connection.\footnote{One also has tensorial curvature in the
  more general setting where $ \nabla $ is an $A$-connection on a
  vector bundle $E \rightarrow M $. The proof is the same, just
  replacing $ Z \in \Gamma (A) $ by $ u \in \Gamma (E) $.} The proof
that
\begin{equation*}
  fR(X,Y)Z = R(fX,Y)Z = R(X,fY)Z = R(X,Y)fZ
\end{equation*}
is by direct calculation, showing that all terms involving the anchor
cancel. For $ R (f X, Y) Z $, one gets the two canceling terms
$ \pm \rho (Y) [f] \nabla _X Z $. Similarly, for $ R(X, fY) $, one
gets the two canceling terms $ \pm \rho (X) [f] \nabla _Y Z
$. Finally, for $ R(X,Y) f Z $, one gets the additional term
$ \bigl( \llbracket \rho (X), \rho (Y) \rrbracket _J - \rho (
\llbracket X, Y \rrbracket ) \bigr) [f] Z $, which vanishes by
\autoref{prop:lieAlgebroidAnchorHomomorphism}. Similarly, one gets
canceling terms $ \pm \rho (Y) [f] X $ when computing $ T(fX, Y) $ and
$ \pm \rho (X) [f] Y $ when computing $ T (X,fY) $, which implies the
tensoriality of $T$.

\subsection{Algebras of $A$-connections}

We next relate the curvature and torsion of an $A$-connection
$ \nabla $ to Lie-admissible, pre-Lie, and post-Lie algebraic
structures on $ \Gamma (A) $ with the product $ \triangleright $. This
generalizes the results of \autoref{sec:affine} on affine connections,
which correspond to the case $ A = T M $. 

\subsubsection{Lie-admissible algebroids}

We begin by introducing Lie-admissible algebroids, which are a natural
generalization of Lie-admissible algebras.

\begin{definition}
  A \emph{Lie-admissible algebroid} $ (A, \rho, \nabla ) $ is an
  anchored bundle $ (A, \rho ) $, equipped with an $A$-connection
  $ \nabla $ on $A$, such that
  $ \bigl( \Gamma (A) , \triangleright \bigr) $ is a Lie-admissible
  algebra.
\end{definition}

\begin{proposition}
  \label{prop:lie-admissible_algebroid_iff_lie_algebroid}
  Let $ ( A, \rho ) $ be an anchored bundle and $ \nabla $ an
  $A$-connection on $A$. Then $ ( A, \rho, \nabla ) $ is a
  Lie-admissible algebroid if and only if $ ( A, \rho ) $ admits a Lie
  algebroid structure such that $ \nabla $ is torsion-free.
\end{proposition}

\begin{proof}
  The condition that $ ( A, \rho ) $ admits a Lie algebroid structure
  such that $ \nabla $ is torsion-free simply says that
  $ \bigl( A, \rho, \llbracket \cdot , \cdot \rrbracket \bigr) $ is a
  Lie algebroid, where
  $ \llbracket X, Y \rrbracket \coloneqq X \triangleright Y - Y
  \triangleright X $ is the commutator bracket.
  
  First, if
  $ \bigl( A, \rho, \llbracket \cdot , \cdot \rrbracket \bigr) $ is a
  Lie algebroid, then by definition,
  $ \llbracket \cdot , \cdot \rrbracket $ is a Lie bracket on
  $ \Gamma (A) $, so \autoref{prop:lieadmissible_bracket} implies
  Lie-admissibility.

  Conversely, if $ (A, \rho, \nabla ) $ is Lie-admissible, then
  \autoref{prop:lieadmissible_bracket} implies that
  $ \llbracket \cdot , \cdot \rrbracket $ is a Lie bracket, so it
  suffices to show that it satisfies the Leibniz rule
  \eqref{eqn:lieAlgebroidLeibniz}. Indeed,
  \begin{align*}
    \llbracket X, f Y \rrbracket
    &\coloneqq X \triangleright (f Y ) - (f Y ) \triangleright X\\
    &= \rho (X) [f] Y + f (X \triangleright Y ) - f (Y \triangleright X) \\
    &= \rho (X) [f] Y + f \llbracket X, Y \rrbracket ,
  \end{align*}
  which completes the proof.
\end{proof}

\begin{example}
  A Lie-admissible algebra is just a Lie-admissible algebroid over a
  point; \autoref{prop:lie-admissible_algebroid_iff_lie_algebroid}
  gives the corresponding Lie algebra as a Lie algebroid over a point.
  More generally, a Lie-admissible algebroid with trivial anchor can
  be seen as a ``bundle of Lie-admissible algebras,'' and
  \autoref{prop:lie-admissible_algebroid_iff_lie_algebroid} gives
  the corresponding bundle of Lie algebras.
\end{example}

The next results examine the situation where
$ \bigl( A, \rho, \llbracket \cdot , \cdot \rrbracket \bigr) $ is a
given Lie algebroid, whose bracket is not \emph{a priori} equal to the
commutator of $ \triangleright $.

\begin{proposition}
  \label{prop:lie-admissible_algebroid_torsion}
  Let $ \bigl( A, \rho, \llbracket \cdot , \cdot \rrbracket \bigr) $
  be a Lie algebroid and $ \nabla $ be an $A$-connection on $A$. If
  $ ( A, \rho, \nabla ) $ is Lie-admissible, then
  $ \rho \circ T = 0 $.
\end{proposition}

\begin{proof}
  If $ ( A, \rho, \nabla ) $ is Lie-admissible, then
  \autoref{prop:lie-admissible_algebroid_iff_lie_algebroid} implies
  that we have two Lie algebroid structures: one with
  $ \llbracket \cdot , \cdot \rrbracket $, and the other with the
  commutator bracket. However,
  \autoref{prop:lieAlgebroidAnchorHomomorphism} implies that $ \rho $
  maps each of these to the Jacobi--Lie bracket on
  $ \mathfrak{X} (M) $, so
  \begin{equation*}
    \rho ( X \triangleright Y - Y \triangleright X ) = \bigl\llbracket \rho (X) , \rho (Y) \bigr\rrbracket _J = \rho \bigl( \llbracket X, Y \rrbracket \bigr) .
  \end{equation*}
  Hence, $ \rho \bigl( T(X,Y) \bigr) = 0 $, for all
  $ X, Y \in \Gamma (A) $.
\end{proof}

Unlike the situation for affine connections in
\autoref{prop:affine_lie-admissible_jacobi-lie}, we may not
necessarily conclude that $ \llbracket \cdot , \cdot \rrbracket $
agrees with the commutator bracket. However, we \emph{may} conclude
this if the anchor is injective.

\begin{corollary}
  Let $ \bigl( A, \rho, \llbracket \cdot , \cdot \rrbracket \bigr) $
  be a Lie algebroid and $ \nabla $ be an $A$-connection on $A$. If
  $ \nabla $ is torsion-free, then $ ( A , \rho , \nabla ) $ is
  Lie-admissible. The converse is true if $ \rho $ is injective.
\end{corollary}

\begin{proof}
  If $ T = 0 $, then the commutator bracket of $ \triangleright $ is
  precisely $ \llbracket \cdot , \cdot \rrbracket $, so
  \autoref{prop:lie-admissible_algebroid_iff_lie_algebroid} implies
  that $ ( A, \rho, \nabla ) $ is Lie-admissible.

  Conversely, if $ (A, \rho, \nabla ) $ is Lie-admissible, then
  \autoref{prop:lie-admissible_algebroid_torsion} says that
  $ \rho \circ T = 0 $, which implies $ T = 0 $ under the assumption
  that $ \rho $ is injective.
\end{proof}

\begin{remark}
  \autoref{th:affineconnect}\ref{th:affineconnect_lie-admissible} is a
  special case of this result for $ A = T M $, where the anchor
  $ \rho = \mathrm{id} _{ T M } $ is injective.
\end{remark}

The following counterexample shows that the converse above is
generally not true unless $\rho$ is injective.

\begin{example}
  \label{ex:bla_trivial_connection}
  Consider a bundle of Lie algebras
  $ \bigl( A, 0 , \llbracket \cdot , \cdot \rrbracket \bigr) $. Since
  the anchor is trivial, we may take the trivial connection
  $ \nabla _X Y = 0 $. This is clearly Lie-admissible, but its torsion
  $ T(X,Y) = - \llbracket X, Y \rrbracket $ generally does not vanish.
\end{example}

We may also obtain necessary and sufficient geometric conditions for
Lie-admissibility, in cases where $\rho$ is not injective, by imposing
some mild restrictions on the $A$-connection $ \nabla $. In the next
proposition, we assume that $ \nabla _X Y = 0 $ whenever
$ \rho (X) = 0 $. This is always the case, for instance, when
$ \nabla $ arises from a $ TM $-connection on $A$ using the
construction in \autoref{prop:tm-connection_a-connection}.

\begin{proposition}
  Let $ \bigl( A, \rho , \llbracket \cdot , \cdot \rrbracket \bigr) $
  be a Lie algebroid and $ \nabla $ an $A$-connection on $A$ such that
  $ \nabla _X Y = 0 $ whenever $ \rho (X) = 0 $. Then
  $ ( A, \rho, \nabla ) $ is Lie-admissible if and only if
  $ \rho \circ T = 0 $ and
  \begin{equation*}
    \csum R(X,Y) Z = 0 ,
  \end{equation*}
  for all $ X, Y, Z \in \Gamma (A) $.
\end{proposition}

\begin{proof}
  Using \autoref{prop:triple}, we obtain
  \begin{equation*}
    R(X,Y) Z = [X,Y,Z]+ T(X,Y) \triangleright Z .
  \end{equation*}
  If $ \rho \circ T = 0 $, then the assumption on $ \nabla $ gives
  $ T(X,Y) \triangleright Z = 0 $, so
  \begin{equation*}
    R(X,Y) Z = [X,Y,Z] .
  \end{equation*}
  The result then follows immediately from the definition of
  Lie-admissibility, together with
  \autoref{prop:lie-admissible_algebroid_torsion}.
\end{proof}

Finally, note that every Lie algebroid admits an $A$-connection
$ \nabla $ (pick any $ T M $-connection on $A$ and apply
\autoref{prop:tm-connection_a-connection}) and thus admits a
torsion-free $A$-connection $ \widetilde{ \nabla } $. Therefore, as
with the case of affine connections, Lie-admissibility does not
actually reveal any information about the Lie algebroid itself.

\subsubsection{Pre-Lie algebroids}

We next introduce what we call pre-Lie algebroids, which are a natural
generalization of pre-Lie algebras to the algebroid
setting.\footnote{We caution the reader that the term ``pre-Lie
  algebroid'' has occasionally appeared in the literature
  \citep{GrUr1997} to mean an almost-Lie algebroid, i.e., an algebroid
  where $ \llbracket \cdot , \cdot \rrbracket $ is not required to
  satisfy the Jacobi identity \citep{GrJo2011}. This is different from
  our definition.}

\begin{definition}
  A \emph{pre-Lie algebroid} $ (A, \rho, \nabla ) $ is an anchored
  bundle $ (A, \rho ) $, with an $A$-connection $ \nabla $ on $A$,
  such that $ \bigl( \Gamma (A) , \triangleright \bigr) $ is a pre-Lie
  algebra.
\end{definition}

From this definition, we immediately see that every pre-Lie algebroid
is Lie-admissible, so the results of the previous section apply.

\begin{proposition}
  \label{prop:pre-lie_algebroid_iff_flat_and_torsion-free}
  Let $ ( A, \rho ) $ be an anchored bundle and $ \nabla $ an
  $A$-connection on $A$. Then $ ( A, \rho, \nabla ) $ is a pre-Lie
  algebroid if and only if $ ( A, \rho ) $ admits a Lie algebroid
  structure such that $ \nabla $ is flat and torsion-free.
\end{proposition}

\begin{proof}
  If $ (A, \rho, \nabla ) $ is pre-Lie, then in particular it is
  Lie-admissible. Therefore,
  \autoref{prop:lie-admissible_algebroid_iff_lie_algebroid} implies
  that $ \bigl( A, \rho, \llbracket \cdot , \cdot \rrbracket \bigr) $
  is a Lie algebroid, where $ \llbracket \cdot , \cdot \rrbracket $ is
  the commutator of $ \triangleright $, with respect to which
  $ \nabla $ is torsion-free. As in the proof of
  \autoref{prop:preLie}, applying \autoref{prop:triple} with $ T = 0 $
  gives $ R(X,Y) Z = [X,Y,Z] = 0 $, so the connection is also flat,
  and the converse direction is immediate from \autoref{prop:triple}
  with $ R = 0 $ and $ T = 0 $.
\end{proof}

\begin{proposition}
  \label{prop:pre-Lie_algebroid_iff_flat_and_rho_torsion}
  Let $ \bigl( A, \rho , \llbracket \cdot , \cdot \rrbracket \bigr) $
  be a Lie algebroid and $ \nabla $ an $A$-connection on $A$ such that
  $ \nabla _X Y = 0 $ whenever $ \rho (X) = 0 $. Then
  $ ( A, \rho, \nabla ) $ is pre-Lie if and only if $ R = 0 $ and
  $ \rho \circ T = 0 $.
\end{proposition}

\begin{proof}
  If $ ( A, \rho, \nabla ) $ is pre-Lie, then in particular it is
  Lie-admissible, so \autoref{prop:lie-admissible_algebroid_torsion}
  implies that $ \rho \circ T = 0 $.  From the assumption on
  $ \nabla $, we have $ T ( X, Y ) \triangleright Z = 0 $, so
  \autoref{prop:triple} implies $ R(X,Y) Z = [X,Y,Z] = 0 $.
  Conversely, if $ R = 0 $ and $ \rho \circ T = 0 $, then again the
  assumption on $ \nabla $ gives $ T ( X, Y ) \triangleright Z = 0 $,
  so \autoref{prop:triple} implies $ [X,Y,Z] = 0 $.
\end{proof}

\begin{remark}
  If $\rho$ is injective, then the condition $ \rho \circ T = 0 $ in
  \autoref{prop:pre-Lie_algebroid_iff_flat_and_rho_torsion} is
  equivalent to $ T = 0 $. In particular,
  \autoref{th:affineconnect}\ref{th:affineconnect_pre-lie} becomes a
  special case of this result for $ A = T M $, since the anchor
  $ \rho = \mathrm{id} _{ T M } $ is injective.
\end{remark}

\begin{example}
  \label{ex:bla_pre-Lie}
  Recall, from \autoref{ex:bla_trivial_connection}, that if
  $ \bigl( A, 0 , \llbracket \cdot , \cdot \rrbracket \bigr) $ is a
  bundle of Lie algebras with $ \nabla $ the trivial connection, then
  $ ( A, \rho, \nabla ) $ is a Lie-admissible algebroid whose torsion
  $ T(X,Y) = - \llbracket X, Y \rrbracket $ generally does not
  vanish. In fact, this is also a pre-Lie algebroid, since the fact
  that the connection is trivial immediately gives $ [X,Y,Z] = 0
  $. Hence, $ T = 0 $ is generally not a necessary condition for a
  pre-Lie algebroid, unless $ \rho $ is injective.
\end{example}

\subsubsection{Post-Lie algebroids} Unlike the definitions of
Lie-admissible and pre-Lie algebroids above, which to our knowledge
are new, the definition of a post-Lie algebroid appeared in
\citet{MuLu2013}.

\begin{definition}
  A \emph{post-Lie algebroid}
  $ \bigl( A, \rho, [ \cdot , \cdot ] , \nabla \bigr) $ is an anchored
  bundle $ ( A, \rho ) $ with a tensorial Lie bracket
  $ [ \cdot , \cdot ] $ on $ \Gamma (A) $ and an $A$-connection
  $ \nabla $, such that
  $ \bigl( \Gamma (A) , [ \cdot , \cdot ] , \triangleright \bigr) $ is
  a post-Lie algebra.
\end{definition}

\citet[Proposition 2.24]{MuLu2013} showed that a Lie algebroid
equipped with a flat and torsion-free connection admits a post-Lie
algebroid structure. The following theorem, which is the main result
of this section, strengthens this by providing both necessary and
sufficient conditions for a post-Lie structure.

\begin{theorem}
  \label{prop:postLieIffCurvaturesVanish}
  Let $ ( A, \rho ) $ be an anchored bundle and $ \nabla $ an
  $A$-connection on $A$. Then $ (A, \rho ) $ admits a post-Lie
  algebroid structure
  $ \bigl( A, \rho, [ \cdot , \cdot ] , \nabla \bigr) $ if and only if
  it admits a Lie algebroid structure
  $ \bigl( A, \rho, \llbracket \cdot , \cdot \rrbracket \bigr) $ such
  that $ R = \overline{ R } = 0 $.
\end{theorem}

\begin{proof}
  If $ \bigl( A, \rho, [ \cdot . \cdot ] , \nabla \bigr) $ is a
  post-Lie algebroid, then \autoref{prop:postLieBracket} implies that
  $ \llbracket X, Y \rrbracket \coloneqq X \triangleright Y - Y
  \triangleright X + [ X, Y ] $ is a Lie bracket on $ \Gamma (A)
  $. Moreover, for all $ f \in C ^\infty (M) $,
  \begin{align*} 
    \llbracket X, f Y \rrbracket
    &= X \triangleright ( f Y ) - ( f Y ) \triangleright X + [ X, f Y ] \\
    &= \rho (X) [f] Y + f ( X \triangleright Y ) - f ( Y \triangleright X ) + f [ X, Y ] \\
    &= \rho (X) [f] Y + f \llbracket X, Y \rrbracket ,
  \end{align*}
  so the Leibniz rule \eqref{eqn:lieAlgebroidLeibniz} holds, and hence
  $ \bigl( A, \rho , \llbracket \cdot , \cdot \rrbracket \bigr) $ is a
  Lie algebroid. Now, since $ [ X, Y ] = - T (X, Y) $,
  \autoref{prop:triple} implies
  \begin{equation*}
    R(X,Y) Z = [X,Y,Z] - [ X,Y] \triangleright Z ,
  \end{equation*}
  which vanishes by the post-Lie condition
  \eqref{eqn:postLieCurvature}. Substituting $ R = 0 $ into
  \eqref{eqn:bianchi1} and using the definition of $ \nabla T $ from
  \eqref{eq:NT} then gives
  \begin{equation*}
    \overline{ R } ( X, Y ) Z = [ Z \triangleright X , Y ] + [ X, Z \triangleright Y ] - Z \triangleright [X,Y]
  \end{equation*}
  which vanishes by the other post-Lie condition
  \eqref{eqn:postLieTorsion}. Hence, $ R = \overline{ R } = 0 $.

  Conversely, suppose
  $ \bigl( A , \rho, \llbracket \cdot , \cdot \rrbracket \bigr) $ is a
  Lie algebroid such that $ R = \overline{ R } = 0 $, and let
  $ [ X, Y ] = - T ( X, Y ) = \overline{ T } ( X, Y ) $, which is
  tensorial. Then \eqref{eqn:bianchi2} implies that
  $ [ \cdot , \cdot ] $ satisfies the Jacobi identity, so in fact this
  is a tensorial Lie bracket. Finally, \eqref{eqn:bianchi1} implies
  the post-Lie condition \eqref{eqn:postLieTorsion}, while $ R = 0 $
  is equivalent to the post-Lie condition
  \eqref{eqn:postLieCurvature}. Hence,
  $ \bigl( A, \rho, [ \cdot , \cdot ] , \nabla \bigr) $ is a post-Lie
  algebroid.
\end{proof}

\begin{remark}
  \autoref{th:affineconnect}\ref{th:affineconnect_post-lie} is a
  special case of this result when $ A = T M $. Together with the
  preceding results, characterizing Lie-admissible and post-Lie
  algebroids in terms of curvature and torsion, we have now completed
  the generalization of \autoref{th:affineconnect} to the algebroid
  setting.
\end{remark}

\section{Pre-Lie, post-Lie, and action algebroids}
\label{sec:action}

As stated in the introduction, \citet{Munthe-Kaas1999} (see also
\citet{MuWr2008}) showed that Lie--Butcher series methods may be
applied to approximate flows on a manifold $M$ equipped with a
transitive $\mathfrak{g}$-action, where $ \mathfrak{g} $ is a Lie
algebra. This work was motivated by the question of how to construct
and analyze numerical integrators on manifolds more general than Lie
groups. In the language of \citet{MuLu2013} and of this paper, this is
due to the fact that an action algebroid $ \mathfrak{g} \ltimes M $
admits a post-Lie algebroid structure, when equipped with its
canonical flat connection. When $ \mathfrak{g} $ is abelian, this
algebroid is actually pre-Lie, and ordinary Butcher series methods,
such as Runge--Kutta methods, may be used.

In this section, we prove local converses to these statements. Namely,
we prove that every transitive post-Lie algebroid on $M$, whose
$A$-connection arises from a $ T M $-connection, is locally isomorphic
to the action algebroid of a transitive $\mathfrak{g}$-action with its
canonical flat connection---and in the pre-Lie case, $ \mathfrak{g} $
must be abelian. These local isomorphisms are actually global when $M$
is simply connected. Essentially, this shows that there is no other
way of applying Lie--Butcher series methods to $M$, other than by
equipping $M$ with a $\mathfrak{g}$-action.

We note that \citet{Blaom2006} and \citet{AbCr2012} investigated the
question of when a Lie algebroid is (locally) an action algebroid,
dropping the assumption of transitivity but imposing assumptions on
$ \nabla $ that are stronger than the post-Lie condition. (This can be
seen as an alternative way of generalizing the Cartan--Nomizu results
to Lie algebroids.)  Namely, they assume a flat $ T M $-connection on
$A$, which is stronger than $ R = 0 $ for the $A$-connection, and that
it be a \emph{Cartan connection} (in the language of
\citet{Blaom2006}) or have vanishing \emph{basic curvature} (in the
language of \citet{AbCr2012}), which is stronger than
$ \overline{ R } = 0 $. Our proofs adapt some of these techniques
(especially \citet[Proposition 2.12]{AbCr2012}) to the transitive
pre-Lie and post-Lie cases.

In addition to these converse results, we also provide new,
streamlined proofs of some of the forward results that had appeared in
\citet{MuLu2013}, based on the characterizations developed in
\autoref{sec:algebroids} and the tensoriality of the curvature and
torsion.

\subsection{Main results}
\label{sec:main_results}

We begin with the pre-Lie case, characterizing the relationship
between pre-Lie algebroids and abelian action algebroids.

\begin{proposition}
  \label{prop:abelian_action_algebroid_pre-Lie}
  If an abelian Lie algebra $ \mathfrak{g} $ acts on $M$, then the
  action algebroid $ \mathfrak{g} \ltimes M $ admits a pre-Lie
  algebroid structure.
\end{proposition}

\begin{proof}
  Since $ \mathfrak{g} \ltimes M $ is a trivial bundle, take
  $ \nabla $ to be the flat $ T M $-connection on
  $ \mathfrak{g} \ltimes M $, and consider the corresponding
  $ \mathfrak{g} \ltimes M $-connection arising from
  \autoref{prop:tm-connection_a-connection}. Now, since $R$ and $T$
  are tensors, we may evaluate them pointwise by extending to constant
  sections. However, $ \nabla _\xi \eta = 0 $ and
  $ \llbracket \xi, \eta \rrbracket = 0 $ for all constant sections
  $ \xi, \eta \in \mathfrak{g} $, so $R$ and $T$ vanish. Hence,
  \autoref{prop:pre-lie_algebroid_iff_flat_and_torsion-free} implies
  that this is a pre-Lie algebroid.
\end{proof}

\begin{theorem}
  \label{thm:pre-Lie_iff_abelian_action}
  Let $ ( A, \rho ) $ be a transitive anchored bundle over $M$ and
  $ \nabla $ be a $ TM $-connection on $A$. Then
  $ ( A, \rho, \nabla ) $ is a pre-Lie algebroid if and only if
  $ \bigl( A, \rho, \llbracket \cdot , \cdot \rrbracket \bigr) $ is
  locally isomorphic to the action algebroid of a transitive abelian
  Lie algebra action on $M$, with $ \nabla $ locally the canonical
  flat connection. This isomorphism is global if $M$ is simply
  connected.
\end{theorem}

\begin{proof}
  The converse follows from the argument in
  \autoref{prop:abelian_action_algebroid_pre-Lie}, with the minor
  modification that we evaluate $R$ and $T$ by extending to
  \emph{locally} constant sections. It only remains to prove the
  forward direction.

  Suppose $ (A, \rho, \nabla ) $ is a pre-Lie algebroid. By
  \autoref{prop:pre-lie_algebroid_iff_flat_and_torsion-free}, the
  $A$-connection $ \nabla $ is flat with respect to the commutator
  bracket $ \llbracket \cdot , \cdot \rrbracket $. Now, denoting by
  $ R _{ T M } $ the curvature of the $ T M $-connection, it is
  straightforward to see that
  \begin{equation*}
    R ( X, Y ) Z = R _{ T M } \bigl( \rho (X) , \rho (Y) \bigr) Z ,
  \end{equation*}
  for all $ X, Y , Z \in \Gamma (A) $. Since $ \rho $ is surjective,
  $ R = 0 $ implies $ R _{ T M } = 0 $, so $ \nabla $ is a flat
  $ T M $-connection on $A$.

  Therefore, we may take a local (or global, if $M$ is simply
  connected) frame of $ \nabla $-flat sections $ e _1, \ldots, e _n
  $. In particular, $ \nabla _{ e _i } e _j = 0 $ for all
  $ i , j = 1, \ldots, n $.  However, since
  $ \llbracket \cdot , \cdot \rrbracket $ is the commutator bracket,
  we have
  \begin{equation*}
    \llbracket e _i , e _j \rrbracket = \nabla _{ e _i } e _j - \nabla _{ e _j } e _i = 0 .
  \end{equation*} 
  Thus,
  $ \mathfrak{g} = \operatorname{span} \{ e _1, \ldots, e _n \} $ is
  an abelian Lie algebra, and $ \rho $ is a $ \mathfrak{g} $-action.
\end{proof}

We next consider the post-Lie case, where $ \mathfrak{g} $ is
generally nonabelian.

\begin{proposition}
  \label{prop:action_algebroid_post-Lie}
  Every action algebroid admits a post-Lie algebroid structure.
\end{proposition}

\begin{proof}
  As in \autoref{prop:abelian_action_algebroid_pre-Lie},
  $ \mathfrak{g} \ltimes M $ is a trivial bundle, so take the flat
  $ T M $-connection and its corresponding
  $ \mathfrak{g} \ltimes M $-connection $ \nabla $. Since $R$ and
  $ \overline{ R } $ are tensors, we may evaluate them pointwise by
  extending to constant sections. However, $ \nabla _\xi \eta = 0 $
  for all constant sections $ \xi, \eta \in \mathfrak{g} $, so
  $R = 0 $ trivially and $ \overline{ R } = 0 $ by the Jacobi
  identity. The result follows by
  \autoref{prop:postLieIffCurvaturesVanish}.
\end{proof}

\begin{theorem}
  \label{thm:post-Lie_iff_action}  
  Let $ ( A, \rho ) $ be a transitive anchored bundle over $M$ and
  $ \nabla $ be a $ TM $-connection on $A$. Then
  $ \bigl( A, \rho, [ \cdot, \cdot ] , \nabla \bigr) $ is a post-Lie
  algebroid if and only if
  $ \bigl( A, \rho, \llbracket \cdot , \cdot \rrbracket \bigr) $ is
  locally isomorphic to the action algebroid of a transitive Lie
  algebra action on $M$, with $ \nabla $ locally the canonical flat
  connection. This isomorphism is global if $M$ is simply connected.
\end{theorem}

\begin{proof}
  The proof is similar in spirit to that of
  \autoref{thm:pre-Lie_iff_abelian_action}.  The converse follows from
  the argument in \autoref{prop:action_algebroid_post-Lie}, with the
  minor modification that we evaluate $R$ and $T$ by extending to
  \emph{locally} constant sections. It only remains to prove the
  forward direction.

  Suppose $ \bigl( A, \rho, [ \cdot , \cdot ] , \nabla \bigr) $ is a
  post-Lie algebroid. By \autoref{prop:postLieIffCurvaturesVanish},
  the $A$-connection $ \nabla $ satisfies $ R = \overline{ R } = 0 $
  with respect to the Lie algebroid structure defined by the bracket
  $ \llbracket X, Y \rrbracket \coloneqq X \triangleright Y - Y
  \triangleright X + [ X, Y ] $. As in the proof of
  \autoref{thm:pre-Lie_iff_abelian_action}, $ R = 0 $ together with
  surjectivity of $\rho$ implies that the $ T M $-connection
  $ \nabla $ is flat.

  Therefore, take a local (or global, if $M$ is simply connected)
  frame of $ \nabla $-flat sections $ e _1 , \ldots, e _n $, and
  define the structure functions $ c _{ i j } ^k \in C ^\infty (M) $
  such that
  $ \llbracket e _i , e _j \rrbracket = \sum _{ k = 1 } ^n c _{ i j }
  ^k e _k $. Since these sections are $ \nabla $-flat, for any
  $ i , j = 1, \ldots, n $ and $ X \in \Gamma (A) $, we have
  \begin{gather*}
    \overline{ \nabla } _{ e _i } \overline{ \nabla } _{ e _j } X = \bigl\llbracket e _i , \llbracket e _j , X \rrbracket \bigr\rrbracket , \qquad 
    \overline{ \nabla } _{ e _j } \overline{ \nabla } _{ e _i } X = - \bigl\llbracket e _j , \llbracket X, e _i \rrbracket \bigr\rrbracket ,\\
    \overline{ \nabla } _{ \llbracket e _i , e _j \rrbracket } X = \nabla _X \llbracket e _i , e _j \rrbracket - \bigl\llbracket X, \llbracket e _i, e _j \rrbracket \bigr\rrbracket = \sum _{ k = 1 } ^n  \rho (X) [ c _{ i j } ^k ] e _k - \bigl\llbracket X, \llbracket e _i, e _j \rrbracket \bigr\rrbracket.
  \end{gather*}
  These expressions, together with the Jacobi identity, immediately
  give
  \begin{equation*}
    \overline{ R } ( e _i , e _j ) X = \sum _{ k = 1 } ^n \rho (X) [ c _{ i j } ^k ] e _k ,
  \end{equation*} 
  so $ \overline{ R } = 0 $ implies that
  $ \rho (X) [ c _{ i j } ^k ] = 0 $ for all
  $ i, j, k = 1, \ldots, n $ and $ X \in \Gamma (A) $. Since $ \rho $
  is surjective, this implies that the structure functions
  $ c _{ i j } ^k $ are in fact constants. Therefore,
  $ \mathfrak{g} = \operatorname{span} \{ e _1, \ldots, e _n \} $ is a
  Lie algebra with structure constants $ c _{ i j } ^k $, and $ \rho $
  is a $ \mathfrak{g} $-action.
\end{proof}

\subsection{Remarks on the non-transitive case}

The transitivity assumption is important for numerical integration and
analysis of flows on $M$ using Butcher and Lie--Butcher series
methods.  Transitivity allows us to locally ``lift'' vector fields on
$M$ to sections of $A$, apply these methods using the pre-Lie or
post-Lie structure of $\Gamma (A) $, and then drop back down to $M$.

However, recall that when the Lie algebroid $ A \rightarrow M $ is
non-transitive, the anchor induces a (generally singular) foliation of
$M$ into leaves $ L \subset M $, and the restriction
$ A _L \rightarrow L $ is a transitive Lie algebroid on each leaf. In
this case, we can only ``lift'' vector fields on $M$ that are tangent
to leaves of the foliation, so it is sufficient to restrict to each
leaf and apply the results of \autoref{sec:main_results}.

The following example illustrates that the results of
\autoref{sec:main_results} generally do not hold in the non-transitive
setting, although they do hold leaf-by-leaf.

\begin{example}
  Let $ A = T \mathbb{S} ^2 \rightarrow \mathbb{S} ^2 $, but take
  $ \rho = 0 $ instead of the identity. For any affine connection
  $ \nabla $, the induced $A$-connection is trivial, so
  $(A, \rho, \nabla )$ is a pre-Lie algebroid. In this case, the
  commutator bracket $ \llbracket \cdot , \cdot \rrbracket $ is also
  trivial, so
  $ \bigl( A , \rho , \llbracket \cdot , \cdot \rrbracket \bigr) $ is
  a bundle of abelian Lie algebras over $ \mathbb{S} ^2 $, where each
  fiber is isomorphic as a Lie algebra to $ \mathbb{R}^2 $. (Since the
  fibers are all isomorphic, this is actually something stronger than
  a bundle of Lie algebras: it is a so-called \emph{Lie algebra
    bundle}.)  However, this is not isomorphic---even locally---to the
  trivial action algebroid $ \mathbb{R}^2 \ltimes \mathbb{S} ^2 $ with
  $ \nabla $ its canonical flat connection, since $ \mathbb{S} ^2 $
  does not admit a flat affine connection.

  However, since the anchor is trivial, the leaves of the induced
  foliation are just points $ x \in \mathbb{S} ^2 $. The transitive
  Lie algebroids obtained by restricting to leaves are the abelian Lie
  algebra fibers $ A _x \rightarrow \{ x \} $, and of course, each of
  these is isomorphic to the trivial action algebroid
  $ \mathbb{R}^2 \ltimes \{ x \} $.
\end{example}

Finally, we again mention that the results of
\citet{Blaom2006,AbCr2012} show that $A$ is locally isomorphic to an
action algebroid, even without assuming transitivity, when the
connection satisfies stronger assumptions than the pre-Lie or post-Lie
conditions. That this condition is \emph{strictly} stronger is
illustrated by the counterexample above: in this case, $A$ admits a
pre-Lie structure but generally not a connection of the type
considered by \citet{Blaom2006,AbCr2012}.

\section{Conclusion}

We have characterized Lie-admissible, pre-Lie, and post-Lie algebras
of connections in terms of the curvature and torsion of these
connections. For affine connections on a manifold $M$, we related
pre-Lie and post-Lie structures to classical results of
\citet{Cartan1927} and \citet{Nomizu1954} on manifolds admitting flat
affine connections with vanishing or parallel torsion. In the more
general setting of connections on a transitive Lie algebroid over $M$,
we showed that pre-Lie and post-Lie structures may only arise, locally
(or globally, if $M$ is simply connected), from the action algebroid
$ \mathfrak{g} \ltimes M $ of a transitive $\mathfrak{g}$-action on
$M$, equipped with its canonical flat connection. This generalizes the
Cartan--Nomizu results stated above, which correspond to the special
case $ A = T M $. Furthermore, it implies that the approach of
\citet{Munthe-Kaas1999}, which equips $M$ with a transitive
$\mathfrak{g}$-action and applies (Lie--)Butcher series methods, is
essentially the only way to use this family of methods for numerical
integration on manifolds.

Finally, we remark that \citet{Nomizu1954} also considered invariant
affine connections with parallel (but not necessarily vanishing)
curvature and either vanishing or parallel torsion. Manifolds
admitting such connections are locally representable as symmetric
homogeneous spaces (for vanishing torsion) or reductive homogeneous
spaces (for parallel torsion). These do not fit into the pre-Lie or
post-Lie algebraic framework. For symmetric spaces, the appropriate
algebraic objects are \emph{Lie triple systems}
(\citet{Jacobson1949,Loos1969,Helgason2001}), which were used for
numerical integration on symmetric spaces in
\citet{MuQuZa2014}. Forthcoming work in progress studies algebras of
connections such that the triple bracket $ [ \cdot , \cdot , \cdot ] $
gives rise to a Lie triple system.

\subsection*{Acknowledgments}

Ari Stern was supported in part by a grant from the Simons Foundation
(\#279968).

\end{document}